\documentclass{amsart}
\newcommand{\B}{\mathcal{B}}

\newtheorem{theorem}{Theorem}[section]
\newtheorem{lemma}[theorem]{Lemma}

\newtheorem{proposition}[theorem]{Proposition}

\theoremstyle{definition}

\newtheorem{example}[theorem]{Example}

\theoremstyle{remark}
\newtheorem{remark}[theorem]{Remark}
\numberwithin{equation}{section}
\usepackage{color,graphicx,amssymb}
\usepackage[multiple]{footmisc}
\begin{document}
\title{A.C.I.M FOR RANDOM INTERMITTENT MAPS : EXISTENCE, UNIQUENESS AND STOCHASTIC STABILITY}
\author{YUEJIAO DUAN }
    %Address of record for the research reported here
    \address{Department of Mathematical Sciences, Loughborough University,
Loughborough, Leicestershire, LE11 3TU, UK}
\email{Y.Duan@lboro.ac.uk}

\date{\today}
%\dedicatory{This paper is dedicated to our authors.}
\keywords{Intermittent maps, Absolutely Continuous Invariant Measure, Stochastic stability}
\begin{abstract}
 We study a random map $T$ which consists of  intermittent maps $\{T_{k}\}_{k=1}^{K}$ and a position dependent probability distribution $\{p_{k,\varepsilon}(x)\}_{k=1}^{K}$. We prove existence of a unique absolutely continuous invariant measure (ACIM) for the random map $T$. Moreover, we show that, as $\varepsilon$ goes to zero, the invariant density of the random system $T$ converges in the $L^{1}$-norm to the invariant density of the deterministic intermittent map $T_{1}$. The outcome of this paper contains a first result on stochastic stability, in the strong sense, of intermittent maps.

\end{abstract}
\maketitle
\pagestyle{myheadings}
\markboth{A.C.I.M FOR RANDOM INTERMITTENT MAPS}{Y. DUAN}

\section{introduction}
Expanding maps of the interval which admit an indifferent fixed point are good testing tools for physical systems with intermittent behaviour. In \cite{P}, Pianigiani proved the existence of ACIM for a certain class of intermittent maps of the interval. Later, polynomial decay of correlations was proved for such systems independently in \cite{LSV,Y2}. More recently, Hu and Vaienti generalized these results to general higher dimensional systems \cite{HV}.
\vskip 10pt

We are interested in perturbations of intermittent maps. In particular when the indifferent fixed point persists under perturbations. Results on statistical stability of intermittent maps with perturbations of this type were obtained in \cite{A,AV}. More recently results on metastability
\footnote{By a metastable system, we mean a system which initially has at least two ACIMs, but once it is perturbed it admits a unique ACIM. Such models were first studied in the expanding case in \cite{GHW}.}of intermittent maps where the neutral fixed point persists under deterministic perturbations were obtained in \cite{BV}. All the results of \cite{A,AV,BV} are concerned with deterministic perturbations of intermittent maps.
\vskip 10pt

In the random setting, i.e., when a system is randomly perturbed, if the random system admits an ACIM $\mu_{R}$ which converges in the weak$-\ast-$topology to an ACIM $\mu$ of the initial system, then we say that the system is stochastically stable in the weak sense. In addition, if the
density $f_{R}^{\ast}$ of $\mu_{R}$ converges to $f^{\ast},$ the density of $\mu,$ in the $L^{1}-$norm, we say the system is stochastically stable in the strong sense.
\vskip 10pt

In \cite{AT} it was proved that intermittent maps of the type studied in \cite{LSV} are stochastically stable in the weak sense. However, there are no results on the strong stochastic stability of such maps.
\vskip 10pt

In this paper, we study a random map $T$ which consists of a collection of intermittent maps $\{T_{k}\}_{k=1}^{K}$ and a probability distribution $\{p_{k,\varepsilon}(x)\}_{k=1}^{K}$. We prove existence of a unique ACIM for the random map $T$. Moreover, we show that, as $\varepsilon$ goes to zero, the invariant density of the random system $T$ converges in the $L^{1}$-norm to the invariant density of $T_{1}$. We obtain our results by using a cone technique. This cone was also used in \cite{RM} to study Ulam approximations for deterministic intermittent map.
\vskip 10pt

In section 2, we present the setup of the problem. Section 3 contains the proof of the existence and uniqueness of the ACIM for the random map. Our main result in this section is Theorem \ref{thm3.1}. Section 4 contains an example of a random map which satisfy our conditions. In section 5, we show that our random maps give rise to an interesting family of 2-dimensional non-uniformly expanding maps  which admit a unique ACIM. Section 6 contains the stochastic stability result. Our main result in this section is Theorem \ref{thm5.1}.

\section{preliminaries}
\subsection{Setup}
Let $(I,\B(I),m)$ be the measure space, where $I=[0,1], \mathcal{B}(I)$ is Borel $\sigma-$algebra and $m$ is Lebesgue measure.
To simplify the notation in the proofs, we consider a random map which consists of two maps.
The proofs for any finite number of maps is similar. We study a position dependent random map
$$T=\{T_{1}(x),T_{2}(x); p_{1}(x), p_{2}(x)\}, \text{ where }$$
$$T_{1}=\begin{cases}
       x(1+2^{\alpha}x^{\alpha}) \quad x\in[0,\frac{1}{2}),\\
       g_{1}(x) \quad \quad \quad x\in[\frac{1}{2},1].
       \end{cases}
\quad\quad
T_{2}=\begin{cases}
       x(1+2^{\beta}x^{\beta}) \quad x\in[0,\frac{1}{2}),\\
       g_{2}(x) \quad \quad \quad x\in[\frac{1}{2},1].
       \end{cases}
,$$
where $0<\beta<\alpha<1,$\footnote{ Note that the assumption that $0<\beta<\alpha<1$ is essential for strong stochastic stability in the strong sense (convergence in $L^{1}$). If for instance $\alpha > 1$, then $T_{1}$ will admit an infinite invariant measure, i.e. this invariant measure does not have an $L^{1}-$density.}\footnote{The results of this paper hold for the following class of maps: Let $0<\alpha<1.$ $\tau$ satisfying\\
$\bullet$ $\tau(0)=0$ and there is a $t_{0}\in(0,1)$ such that $\tau: [0,t_{0})\rightarrow [0,1), \tau: [t_{0},1]\rightarrow [0,1].$\\
$\bullet$ Each branch of $\tau$ is increasing, convex and is $C^{1}; \tau'(0)=1$ and $\tau'(x)>1$ for all $x\in(0,t_{0})\cup(t_{0},1).$\\
$\bullet$ There is a constant $C\in(0,\infty)$ such that $\tau(x)\geq x+ Cx^{1+\alpha}$ for $x\in[0,t_{0}).$\\
The convexity assumption is essential so that the transfer operator satisfies the cone condition $\int\limits_{0}\limits^{x}fdm\leq Ax^{1-\alpha}m(f).$ We choose to work with a well known representative of this family. Namely, the model studied in \cite{LSV}.}  $g_{k}(\frac{1}{2})=0, g'_{k}(x)>1, k=1,2$ and
$p_{k}:[0,1]\rightarrow [0,1]$ is a measurable function such that $p_{1}(x)+p_{2}(x)=1,$ i.e. $p_{1}(x), p_{2}(x)$ are position dependent probabilities. A position dependent random map is understood as a Markov process with transition function
$$\mathbb{P}(x,A)=p_{1}(x)\chi_{A}(T_{1}(x))+p_{2}(x)\chi_{A}(T_{2}(x)),$$
where $A$ is any measurable set in $\B(I)$
and $\chi_{A}$ is the characteristic function of the set $A.$

\subsection{Invariant measures}
The transition function $\mathbb{P}(x,A)$ induces an operator
$E_{T}$ on measures on $(I,\B(I))$ denoted by
\begin{eqnarray*}
E_{T}\mu(A)
&=&\int\limits_{I} \mathbb{P}(x,A)d\mu(x) \\
&=&\int\limits_{I} p_{1}(x)\chi_{A}(T_{1}(x))+p_{2}(x)\chi_{A}(T_{2}(x))d\mu(x)\\
&=&\int\limits_{T^{-1}_{1}(A)} p_{1}(x)d\mu(x)+\int\limits_{T^{-1}_{2}(A)} p_{2}(x)d\mu(x).
\end{eqnarray*}
We say that $\mu$ is $T-$invariant if and only if
$$E_{T}\mu(A)=\mu(A);$$
that is, for any measurable set $A$,
$$\mu(A)=\int\limits_{T^{-1}_{1}(A)} p_{1}(x)d\mu(x)+\int\limits_{T^{-1}_{2}(A)} p_{2}(x)d\mu(x).$$

\subsection{Transfer operators}
If $\mu$ has a density function $f$ with respect to $m$, then $E_{T}\mu$ has also a density
function which we call $\mathcal{L}_{T}f.$ We obtain, for any measurable set $A$,

\begin{eqnarray}\label{2.2}
\int\limits_{A}\mathcal{L}_{T}f dm(x)
&=&E_{T}\mu(A)=\int\limits_{T^{-1}_{1}(A)} p_{1}(x)d\mu(x)+\int\limits_{T^{-1}_{2}(A)} p_{2}(x)d\mu(x)\nonumber\\
&=&\int\limits_{T^{-1}_{1}(A)} p_{1}(x)fdm(x)+\int\limits_{T^{-1}_{2}(A)} p_{2}(x)fdm(x)\nonumber\\
&=&\int\limits_{A} p_{1}(T^{-1}_{1}(x))f(T^{-1}_{1}(x))dm(x)+
   \int\limits_{A} p_{2}(T^{-1}_{2}(x))f(T^{-1}_{2}(x))dm(x)\nonumber\\
&=&\int\limits_{A}P_{T_{1}}(p_{1}f)dm(x)+\int\limits_{A}P_{T_{2}}(p_{2}f)dm(x)\nonumber\\
&=&\int\limits_{A}[P_{T_{1}}(p_{1}f)+P_{T_{2}}(p_{2}f)]dm(x),
\end{eqnarray}
where $P_{T_{1}}$ and $P_{T_{2}}$ are Perron-Frobenius operators \cite{ABG} associated with $T_{1}$ and $T_{2}$
respectively.
Since \eqref{2.2} holds for any measurable set $A$, we will get an almost everywhere equality:
\footnote{ Note that since $p_{1}(x),p_{2}(x) $ are functions of $x, \mathcal{L}_{T}$ is not a convex combination of $p_{1}$ and $p_{2}.$ }
\begin{eqnarray*}
(\mathcal{L}_{T}f)(x)&=&P_{T_{1}}(p_{1}f)(x)+P_{T_{2}}(p_{2}f)(x)\\
&=&\sum\limits_{y\in T_{1}^{-1}(x)}\frac{(p_{1}f)(y)}{|T_{1}^{'}(y)|}
   +\sum\limits_{y\in T_{2}^{-1}(x)}\frac{(p_{2}f)(y)}{|T_{2}^{'}(y)|}.
\end{eqnarray*}

We call $\mathcal{L}_{T}$ the Perron-Frobenius operator associated with the random map $T.$ The properties of $\mathcal{L}_{T}$ resemble the properties of
the classical Perron-Frobenius operator associated with a single deterministic map. $\mathcal{L}_{T}$ satisfies the properties as follows (See \cite{BG} Lemma 3.1):\\
(i)(Linearity)\quad $\mathcal{L}_{T}: L^{1}\rightarrow L^{1}$ is a linear operator.\\
(ii)(Positivity)\quad Let $f\in L^{1}$ and assume $f\geq 0,$ then $\mathcal{L}_{T}f\geq 0.$\\
(iii)(Preservation of integrals)
$$\int\limits_{I}\mathcal{L}_{T}fdm(x)=\int\limits_{I}fdm(x)$$
(iv)(contraction)\quad for any $f\in L^{1},$
$$\parallel \mathcal{L}_{T}f\parallel_{1}\leq \parallel f\parallel_{1}$$
(v)\quad $\mathcal{L}_{T}f=f\Leftrightarrow E_{T}\mu=\mu$, i.e measure $\mu=f\cdot m$ is
$T-$invariant.\\
(vi)(composition)
$$\mathcal{L}_{T\circ R}f=\mathcal{L}_{T}\circ\mathcal{L}_{R}f$$
In particular, $\mathcal{L}_{T^{n}}f=\mathcal{L}^{n}_{T}f.$

\subsection{Notation}
For $x\in I, k\in \{1,2\}$
and partition $\mathcal{P}=\{I_{1},I_{2}\}, I_{1}=[0,\frac{1}{2}], I_{2}=[\frac{1}{2},1],$
we introduce the following definitions
\begin{eqnarray*}
 T(x)&=&T_{k}(x), \quad \text{ with probability } p_{k}(x)\\
 T^{n}(x)&=&T_{k_{n}}\circ T_{k_{n-1}}\circ\cdot\cdot\cdot\circ T_{k_{1}}(x), \text{ with probability }\\
&& p_{k_{n}}(T_{k_{n-1}}\circ\cdot\cdot\cdot\circ T_{k_{1}}(x))\cdot
                     p_{k_{n-1}}(T_{k_{n-2}}\circ\cdot\cdot\cdot\circ T_{k_{1}}(x))\cdot
                     p_{k_{1}}(x), \quad k_{i}\in \{1,2\}\\
 T_{k,i}&=&T_{k}\mid_{I_{i}}.\\
\end{eqnarray*}
We write
$$T^{-1}_{1}x=\{y_{1}, z_{1}\}, \quad y_{1}\leq \frac{1}{2}\leq z_{1} ;$$
$$T^{-1}_{2}x=\{y_{2}, z_{2}\}, \quad y_{2}\leq \frac{1}{2}\leq z_{2} ;$$
$$y_{\ast}=\max\{y_{1},y_{2}\}\in [0,\frac{1}{2}],\quad z_{\ast}=\max\{z_{1},z_{2}\}\in [\frac{1}{2},1];$$
and $m(f)=\int\limits_{0}\limits^{1}f(x)dm(x),$ where $m$ is Lebesgue measure.

\textbf{Cone.} For $A>0,$ define
$$\mathcal{C}_{A}=\{f\in L^{1}\mid f\geq 0, f \text{ decreasing, } \int\limits_{0}\limits^{x}fdm\leq Ax^{1-\alpha}m(f)\}.$$

\section{existence and uniqueness of acim }
\subsection{Sufficient conditions for the existence of a $T-$ACIM}
For $k=1,2,$ we assume\\
(A) \quad $\sum\limits_{i=1}\limits^{l}\frac{p_{k}(T_{k,i}^{-1}(x))}{T^{'}_{k}(T_{k,i}^{-1}(x))}, 1\leq l\leq 2,$ is decreasing;\\
(B) \quad $\inf\limits_{x\in I}p_{k}(x)\geq \delta >0.$\\

\begin{theorem}\label{thm3.1}\text{Under assumptions (A) and (B)}\\
(i) The random map $T$ admits a unique ACIM $\mu, d\mu=\rho dm.$\\
(ii) The invariant density $\rho$ is uniformly bounded below.
\end{theorem}
We first prove some technical lemmas. The proof of the Theorem \ref{thm3.1} is at the end of this section.

\begin{lemma}\label{le3.2}
Let $f\in \mathcal{C}_{A}.$  Then, for $x\in(0,1],$\\
(i)\quad$f(x)\leq Ax^{-\alpha}m(f);$\\
(ii)\quad $f(x)\leq \frac{1}{x}m(f),$ and in particular, $f(x)|_{x\in[\frac{1}{2},z_{\ast})}\leq 2m(f);$\\
(iii)\quad$y_{1}\geq\frac{x}{2}, y_{2}\geq\frac{x}{2} $ and $x\geq y_{\ast};$\\
(iv)\quad $(1-x)^{1-\alpha}\leq 1-(1-\alpha)x;$\\
(v)\quad$x^{1-\alpha}-y_{\ast}^{1-\alpha}\geq\frac{1-\alpha}{2}x.$
\end{lemma}

\begin{proof}
(i) We have
$$xf(x)=\int\limits_{0}^{x}f(x)dm(\xi)\leq\int\limits_{0}^{x}f(\xi)dm(\xi)
\leq Ax^{1-\alpha}m(f).$$

(ii) By $f(x)\geq 0$ and decreasing, we have
$$xf(x)=\int\limits_{0}^{x}f(x)dm(\xi)\leq \int\limits_{0}^{x}f(\xi)dm(\xi)
\leq \int\limits_{0}^{x}f(\xi)dm(\xi)+\int\limits_{x}^{1}f(\xi)dm(\xi)=m(f).$$
So, $f(x)\leq \frac{1}{x}m(f)$ and in particular $f(x)\leq 2m(f),$ when $x\in[\frac{1}{2},z_{\ast}).$
\vskip 5pt

(iii) For $y_{1},y_{2}\leq \frac{1}{2}, 0<\beta<\alpha<1,$ we have
$$x=T_{1}(y_{1})=y_{1}(1+2^{\alpha}y_{1}^{\alpha})\leq2y_{1} \text{ and }
x=T_{2}(y_{2})=y_{2}(1+2^{\alpha}y_{2}^{\beta})\leq2y_{2}.$$
Also,
$$x=T_{1}(y_{1})=y_{1}(1+2^{\alpha}y_{1}^{\alpha})\geq y_{1} \text{ and }
x=T_{2}(y_{2})=y_{2}(1+2^{\alpha}y_{2}^{\beta})\geq y_{2}.$$
Therefore, $y_{1}\geq\frac{x}{2}, y_{2}\geq\frac{x}{2}$ and
$x\geq y_{\ast}.$
\vskip 5pt

(iv) Set $$g(x)=(1-x)^{1-\alpha}- [1-(1-\alpha)x],$$
then $g(0)=1-1=0$ and for $x\in[0,1],$
 \begin{eqnarray*}
 g'(x)=-(1-\alpha)(1-x)^{-\alpha}+(1-\alpha)=(1-\alpha)[1-\frac{1}{(1-x)^{\alpha}}]\leq 0.
 \end{eqnarray*}
Therefore, $g(x)\leq0, x\in(0,1],$  that is
$(1-x)^{1-\alpha}\leq 1-(1-\alpha)x.$

(v) First write,
$$x^{1-\alpha}-y_{\ast}^{1-\alpha}=x^{1-\alpha}[1-(\frac{y_{\ast}}{x})^{1-\alpha}]=
x^{1-\alpha}[1-(1-\frac{x-y_{\ast}}{x})^{1-\alpha}].$$
Let $\zeta=\frac{x-y_{\ast}}{x}.$

In case $y_{\ast}=y_{1},$
$$x=T_{1}(y_{1})=y_{1}(1+2^{\alpha}y_{1}^{\alpha})>y_{1}>0,\quad x\leq2y_{1}\quad\text {and }\zeta=\frac{x-y_{1}}{x}\in (0,1].$$
Thus,
 \begin{eqnarray*}
 x^{1-\alpha}-y_{\ast}^{1-\alpha}&=&x^{1-\alpha}[1-(1-\zeta)^{1-\alpha}]\\
 &\geq&x^{1-\alpha}[1-(1-(1-\alpha)\zeta)]\\
 &=&x^{1-\alpha}(1-\alpha)\frac{x-y_{1}}{x}\\
 &=&x^{-\alpha}(1-\alpha)(T_{1}(y_{1})-y_{1})\\
 &=&x^{-\alpha}(1-\alpha)(2^{\alpha}y_{1}^{\alpha+1})\\
 &\geq&(2y_{1})^{-\alpha}(1-\alpha)(2^{\alpha}y_{1}^{\alpha+1})\\
 &=&(1-\alpha)y_{1}\\
 &\geq&\frac{(1-\alpha)}{2}x.
 \end{eqnarray*}

In case $y_{\ast}=y_{2},$
$$x=T_{2}(y_{2})=y_{2}(1+2^{\beta}y_{2}^{\beta})>y_{2}>0,\quad  x\leq2y_{2}\quad\text{and }\zeta=\frac{x-y_{2}}{x}\in (0,1].$$
We have
 \begin{eqnarray*}
 x^{1-\alpha}-y_{\ast}^{1-\alpha}&=&x^{1-\alpha}[1-(1-\zeta)^{1-\alpha}]\\
 &\geq&x^{1-\alpha}[1-(1-(1-\alpha)\zeta)]\\
 &=&x^{1-\alpha}(1-\alpha)\frac{x-y_{2}}{x}\\
 &=&x^{-\alpha}(1-\alpha)(T_{2}(y_{2})-y_{1})\\
 &=&x^{-\alpha}(1-\alpha)(2^{\beta}y_{2}^{\beta+1})\\
 &\geq&(2y_{2})^{-\alpha}(1-\alpha)(2^{\beta}y_{2}^{\beta+1})\\
 &=&(1-\alpha)y_{2}(2y_{2})^{\beta-\alpha}.
 \end{eqnarray*}
Since $0<\beta<\alpha< 1$ and $0\leq2y_{2}\leq1$, we get
$(2y_{2})^{\beta-\alpha}\geq 1.$\\
Thus,
 \begin{eqnarray*}
 x^{1-\alpha}-y_{\ast}^{1-\alpha}\geq(1-\alpha)y_{2}\geq(1-\alpha)\frac{x}{2}.
 \end{eqnarray*}
\end{proof}

\begin{lemma}\label{le3.3}
Let $f\geq 0$ be a decreasing function. Then $\mathcal{L}_{T}f$ is also decreasing.
\end{lemma}

\begin{proof}
See Lemma 3.1 of \cite{BG2}.\footnote{ Note that this Lemma only requires assumption (A) to hold.}
\end{proof}

\begin{proposition}\label{pro3.4}
For $A\geq \frac{4}{1-\alpha}$ the cone $\mathcal{C}_{A}$ is invariant under the action of the operator $\mathcal{L}_{T}.$
\end{proposition}

\begin{proof}
By Lemma \ref{le3.3}, for $f\in\mathcal{C}_{A}$ we know that $\mathcal{L}_{T}f$ is decreasing . Also,
$\mathcal{L}_{T}f\geq 0$ and $m(\mathcal{L}_{T}f)=m(f).$ Therefore we only need to prove that
$$\int\limits_{0}\limits^{x}\mathcal{L}_{T}fdm\leq Ax^{1-\alpha}m(\mathcal{L}_{T})=Ax^{1-\alpha}m(f),$$
when $A\geq A_{\ast}=\frac{4}{1-\alpha}.$
We have
\begin{eqnarray*}
\int\limits_{0}\limits^{x}\mathcal{L}_{T}fdm&=&\int\limits_{0}\limits^{x}P_{T_{1}}(p_{1}f)+P_{T_{2}}(p_{2}f)dm
=\int\limits_{T_{1}^{-1}[0,x]}(p_{1}f)dm+\int\limits_{T_{2}^{-1}[0,x]}(p_{2}f)dm\\
&=&(\int\limits_{0}\limits^{y_{1}}+\int\limits_{\frac{1}{2}}\limits^{z_{1}})(p_{1}f)dm +
    (\int\limits_{0}\limits^{y_{2}}+\int\limits_{\frac{1}{2}}\limits^{z_{2}})(p_{2}f)dm\\
&\leq&(\int\limits_{0}\limits^{y_{\ast}}+\int\limits_{\frac{1}{2}}\limits^{z_{\ast}})(p_{1}f)dm +
    (\int\limits_{0}\limits^{y_{\ast}}+\int\limits_{\frac{1}{2}}\limits^{z_{\ast}})(p_{2}f)dm\\
&=&\int\limits_{0}\limits^{y_{\ast}}(p_{1}+p_{2})fdm+\int\limits_{\frac{1}{2}}\limits^{z_{\ast}}(p_{1}+p_{2})fdm
\leq Ay^{1-\alpha}_{\ast}m(f)+\int\limits_{\frac{1}{2}}\limits^{z_{\ast}}fdm,
\end{eqnarray*}
where $y_{\ast}=\max\{y_{1},y_{2}\}\in [0,\frac{1}{2}],\quad z_{\ast}=\max\{z_{1},z_{2}\}\in [\frac{1}{2},1].$
Since our transformations $T_{k}(x)$ may not be piecewise onto, there are two cases to consider.

In the case 1, $x$ has only one pre-image.
By the Lemma \ref{le3.2} (iii), we get $x\geq y_{\ast}.$ So, $y^{1-\alpha}_{\ast}\leq x^{1-\alpha},$ with $1-\alpha>0.$
Therefore, $\int\limits_{0}\limits^{x}\mathcal{L}_{T}fdm\leq Ay^{1-\alpha}_{\ast}$ for $A>0.$

In the case 2, $x$ has two preimages. From Lemma \ref{le3.2}, we have $f(x)\leq 2m(f), x\in [\frac{1}{2},z_{\ast}].$
Then,

$$\int\limits_{\frac{1}{2}}\limits^{z_{\ast}}fdm\leq \int\limits_{\frac{1}{2}}\limits^{z_{\ast}}2m(f)dm=2(z_{\ast}-\frac{1}{2})m(f).$$
Moreover, we have $g'_{1}(x)>1, g'_{2}(x)>1,$ then
$x=m[0,x]=m\circ T_{k}[\frac{1}{2}, z_{k}]\geq z_{k}-\frac{1}{2}, k=1,2$ i.e.
$x> z_{\ast}- \frac{1}{2}.$ So,
$$\int\limits_{\frac{1}{2}}\limits^{z_{\ast}}fdm< 2xm(f).$$
By the result of Lemma \ref{le3.2} (iv), we obtain that $ x\leq\frac{2}{1-\alpha}(x^{1-\alpha}-y_{\ast}^{1-\alpha}).$
Then, for $A\geq A_{\ast}=\frac{4}{1-\alpha},$
\begin{eqnarray*}
\int\limits_{0}\limits^{x}\mathcal{L}_{T}fdm
< Ay^{1-\alpha}_{\ast}m(f)+\frac{4}{1-\alpha}(x^{1-\alpha}-y_{\ast}^{1-\alpha})m(f)\leq Ax^{1-\alpha}m(f).
\end{eqnarray*}
Therefore, $\mathcal{L}_{T}f\in \mathcal{C}_{A},$ for $f\in\mathcal{C}_{A}$ and $A\geq \frac{4}{1-\alpha}.$
\end{proof}

\begin{remark}
Obviously, if $f\in\mathcal{C}_{A}$ and $A\geq \frac{4}{1-\alpha},$ then
$\mathcal{L}^{n}_{T}f\in \mathcal{C}_{A}, n\geq 1.$
\end{remark}

\begin{remark}
Since $\mathcal{C}_{A}$ is compact and convex, operator $\mathcal{L}_{T}$ has a fixed point
$f_{\ast}\in\mathcal{C}_{A}$ by Proposition \ref{pro3.4} and the Schauder-Tychonoff fixed point theorem of \cite{DS}. Thus,
random map $T$ admits an ACIM .
\end{remark}

%   \begin{theorem}
%   Let $T_{k}, k=1,2$ have a unique a.c.i.m, which is equivalent to Lebesgue measure.
%   Suppose the random map $T=\{T_{1}(x),T_{2}(x); p_{1}(x), p_{2}(x)\}$ has an a.c.i.m, then
%   $T$ has a unique a.c.i.m.
%  \end{theorem}

Let $\mu$ be an ACIM for random map $T.$ Each of the maps $T_{k}$ admits a unique ACIM (See Appendix).
Let $\nu_{1}$ and $\nu_{2}$ be the unique ACIM for
$T_{1}$ and $T_{2}$ respectively.
Let $A_{k}=\textrm{supp}(\nu_{k})$ and $\mathcal{U}_{k}=\bigcup\limits_{j=0}\limits^{\infty}T^{-j}_{k}A_{k}$ be its basin.
For $k= 1, 2, $ we have $A_{k}=\mathcal{U}_{k}=I$ (see Appendix).

\begin{lemma}\label{le3.7}
For $k=1,2, I=A_{k}\subseteq\textrm{supp}(\mu).$
\end{lemma}
\begin{proof}
Since $A_{k}=\mathcal{U}_{k}=I$ for $k= 1,2.$ Then $\mu(A_{k})=\mu(\mathcal{U}_{k})>0.$

Let $B= I\cap\textrm{supp}(\mu),$ then $B\neq \emptyset$ and $\mu(B)>0.$ Since $B$ is subset of $I=A_{k}$ and $A_{k}$ is an invariant set, then
$\bigcup\limits_{i=0}\limits^{\infty}T^{i}_{k}B\subseteq A_{k}.$

Assume $A_{k}\nsubseteq \textrm{supp}(\mu).$ Then $\mu(A_{k}\setminus B)=0.$
Also, $$\mu(A_{k}\setminus B)\geq \mu(\bigcup\limits_{i=0}\limits^{\infty}T^{i}_{k}B\setminus B)=\mu(\bigcup\limits_{i=1}\limits^{\infty}T^{i}_{k}B\setminus B)\geq \mu(T^{i}_{k}B), \quad i=1,2,...$$
So, in this case, $\mu(T^{i}_{k}B)\leq 0, i=1,2,... .$ However this leads to a contradiction because, by condition (B),
\begin{eqnarray*}
\mu(T_{1}B)&=&\int\limits_{T^{-1}_{1}(T_{1}B)}p_{1}d\mu+\int\limits_{T^{-1}_{2}(T_{1}B)}p_{2}d\mu\\
&\geq&\inf\limits_{x\in I}p_{1}(x)\mu(B)+\inf\limits_{x\in I}p_{2}(x)\mu(T^{-1}_{2}(T_{1}B))>0.
\end{eqnarray*}
\text{ and }
\begin{eqnarray*}
\mu(T_{2}B)&=&\int\limits_{T^{-1}_{1}(T_{2}B)}p_{1}d\mu+\int\limits_{T^{-1}_{2}(T_{2}B)}p_{2}d\mu\\
&\geq&\inf\limits_{x\in I}p_{1}(x)\mu(T^{-1}_{1}(T_{2}B))+\inf\limits_{x\in I}p_{2}(x)\mu(B)>0
\end{eqnarray*}
Therefore, $I=A_{k}\subseteq\textrm{supp}(\mu).$
\end{proof}

\begin{proposition}\label{pro3.8}
Let $A\geq A_{\ast}=\frac{4}{1-\alpha}$ and $f\in \mathcal{C}_{A}.$ There are
$\gamma>0, N\in \mathbb{Z}_{+}$ such that $\mathcal{L}^{n}_{T}f\geq\gamma m(f),$ for all $n\geq N,$ where
$\gamma$ and $N$ depend only on $A.$ In particular, if $\rho=\mathcal{L}_{T}\rho$ then $\mu=\rho m$ is equivalent to $m.$
\end{proposition}

\begin{proof}
First by Proposition \ref{pro3.4}, if $A\geq \frac{4}{1-\alpha}, f\in \mathcal{C}_{A},$ then $\mathcal{L}^{n}_{T}f\in \mathcal{C}_{A}.$ So, we have
$$\int\limits_{0}^{x}fdm\leq Ax^{1-\alpha}m(f), \int\limits_{0}^{x}\mathcal{L}^{n}_{T}fdm\leq Ax^{1-\alpha}m(\mathcal{L}^{n}_{T}f).$$
Without loss of generality, we suppose that $m(f)=1.$ Then $m(\mathcal{L}^{n}_{T}f)=m(f)=1.$
Therefore we only need to prove $\mathcal{L}^{n}_{T}f\geq\gamma.$ Fix a small number $0<\sigma<\frac{1}{2},$ such that $A\sigma^{1-\alpha}=\frac{1}{2}.$ Then,
$$\int\limits_{0}^{\sigma}fdm\leq A\sigma^{1-\alpha}=\frac{1}{2}  \text{ and } \int\limits_{\sigma}^{1}fdm=1-\int\limits_{0}^{\sigma}fdm\geq\frac{1}{2}.$$
When $x\in(0,\sigma),$ since $f(x)$ is a decreasing function, we have
$$f(x)\geq f(\sigma)=\frac{\int\limits_{\sigma}^{1}f(\sigma)dm}{1-\sigma}\geq \frac{\int\limits_{\sigma}^{1}f(x)dm}{1-\sigma}\geq \frac{1}{2(1-\sigma)}.$$

Moreover, $\mathcal{L}^{n}_{T}f(x)$ is decreasing. Then it is enough to show that  $\mathcal{L}^{n}_{T}f(1)$ is bounded below away from zero. By (vi) composition property of $\mathcal{L}_{T}$ we have $\mathcal{L}^{n}_{T}f(1)=\mathcal{L}_{T^{n}}f(1).$ We will show that $\mathcal{L}_{T^{n}}f(1)\geq\gamma>0.$
Set $\omega_{n}=\{k_{1},k_{2},...,k_{n}\in\{1,2\}^{n}\},  k_{i}\in\{1,2\}.$
Define $x_{n}=T_{k}^{-1}(x_{n-1})\cap[0,\frac{1}{2}], n\geq 1$ and $ x_{0}=1.$ Obviously, $\{x_{n}\}$ is a strictly decreasing sequence and it converges to 0. Since $\{x_{n}\}$ depends on $\omega_{n},$ we denote $\{x_{n,\omega_{n}}\}=\{x_{n}\}(\omega_{n}).$
With the fixed $\sigma,$ we can find an $N$ such that $\{0, b_{1},b_{2},...,b_{q}\}$ are critical points
of map $T_{\omega_{N}}$ and $$\{b_{1},b_{2}\}=T^{-1}_{k_{1}}(x_{N-1}),\quad \max\limits_{\omega_{N}}x_{N-1,\omega_{N}}\leq\sigma.$$
Then, for all $\omega_{N},$ we have $\mathcal{L}_{T}f(x_{N-1,\omega_{N}})\geq\mathcal{L}_{T}f(\sigma)$ since $\mathcal{L}_{T}f(x)$ is decreasing.
\begin{equation}\label{3.1}
\begin{split}
\mathcal{L}_{T^{N}}f(1)
&=\sum\limits_{\omega_{N}\in\{1,2\}^{N}}\sum\limits_{i=1}\limits^{q}\frac{(p_{\omega_{N}}f)(b_{i})}{T'_{\omega_{N}}(b_{i})}\\
&=\sum\limits_{\omega_{N}\in\{1,2\}^{N}}\sum\limits_{i=1}\limits^{q}\frac{p_{\omega_{N-1}}(T_{k_{1}}(b_{i}))p_{k_{1}}(b_{i})f(b_{i})}
{T'_{\omega_{N-1}}(T_{k_{1}}(b_{i}))T'_{k_{1}}(b_{i})}\\
&\geq\sum\limits_{\omega_{N}\in\{1,2\}^{N}}\sum\limits_{i=1}\limits^{2}\frac{p_{\omega_{N-1}}(T_{k_{1}}(b_{i}))p_{k_{1}}(b_{i})f(b_{i})}
{T'_{\omega_{N-1}}(T_{k_{1}}(b_{i}))T'_{k_{1}}(b_{i})}\\
&=\sum\limits_{\omega_{N}\in\{1,2\}^{N}}\sum\limits_{i=1}\limits^{2}\frac{p_{\omega_{N-1}}(x_{N-1})p_{k_{1}}(T^{-1}_{k_{1},i}x_{N-1})f(T^{-1}_{k_{1},i}x_{N-1})}
{T'_{\omega_{N-1}}(x_{N-1})T'_{k_{1}}(T^{-1}_{k_{1},i}x_{N-1})}\\
&=\sum\limits_{\omega_{N-1}\in\{1,2\}^{N-1}}\sum\limits_{k_{1}=1}\limits^{2}\frac{p_{\omega_{N-1}}(x_{N-1})}{T'_{\omega_{N-1}}(x_{N-1})}
(\sum\limits_{i=1}\limits^{2}\frac{p_{k_{1}}(T^{-1}_{k_{1},i}x_{N-1})f(T^{-1}_{k_{1},i}x_{N-1})}{T'_{k_{1}}(T^{-1}_{k_{1},i}x_{N-1})})\\
&=\sum\limits_{\omega_{N-1}\in\{1,2\}^{N-1}}\frac{p_{\omega_{N-1}}(x_{N-1,\omega_{N}})}{T'_{\omega_{N-1}}(x_{N-1,\omega_{N}})}
   [\sum\limits_{k_{1}=1}\limits^{2}\sum\limits_{i=1}\limits^{2}
   \frac{p_{k_{1}}(T^{-1}_{k_{1},i}x_{N-1,\omega_{N}})f(T^{-1}_{k_{1},i}x_{N-1,\omega_{N}})}{T'_{k_{1}}(T^{-1}_{k_{1},i}x_{N-1,\omega_{N}})}]\\
&=\sum\limits_{\omega_{N-1}\in\{1,2\}^{N-1}}\frac{p_{\omega_{N-1}}(x_{N-1,\omega_{N}})}{T'_{\omega_{N-1}}(x_{N-1,\omega_{N}})}
   [\mathcal{L}_{T}f(x_{N-1,\omega_{N}})]\\
&\geq\sum\limits_{\omega_{N-1}\in\{1,2\}^{N-1}}\frac{p_{\omega_{N-1}}(x_{N-1,\omega_{N}})}{T'_{\omega_{N-1}}(x_{N-1,\omega_{N}})}
   [\mathcal{L}_{T}f(\sigma)].
\end{split}
\end{equation}
We have $\max\limits_{k\in\{1,2\},x\in[0,\frac{1}{2}]}T'_{k}(x)=2+\alpha$ and $f(T^{-1}_{k,1}\sigma)>f(\sigma)\geq\frac{1}{2(1-\sigma)},$
and by condition (B): $\inf p_{k}(x)\geq\delta>0.$ Therefore, from \eqref{3.1} it remains to show that $\mathcal{L}_{T}f(\sigma)>0.$ Indeed,

\begin{eqnarray*}
\mathcal{L}_{T}f(\sigma)&=&\frac{p_{1}(T^{-1}_{k,1}\sigma)f(T^{-1}_{k,1}\sigma)}{T'_{k}(T^{-1}_{k,1}\sigma)}+
                           \frac{p_{2}(T^{-1}_{k,2}\sigma)f(T^{-1}_{k,2}\sigma)}{T'_{k}(T^{-1}_{k,2}\sigma)}\\
&\geq&\frac{p_{1}(T^{-1}_{k,1}\sigma)f(T^{-1}_{k,1}\sigma)}{T'_{k}(T^{-1}_{k,1}\sigma)}\geq\frac{\delta}{2(1-\sigma)(2+\alpha)}>0.
\end{eqnarray*}

Therefore,
\begin{eqnarray*}
\mathcal{L}^{N}_{T}f(x)&\geq&\mathcal{L}_{T^{N}}f(1)
% &\geq&
% \sum\limits_{\omega_{N-1}\in\{1,2\}^{N-1}}\frac{p_{\omega_{N-1}}(x_{N-1,\omega_{N}})}{T'_{\omega_{N-1}}(x_{N-1,\omega_{N}})}[\mathcal{L}_{T}f(\sigma)]\\
\geq\gamma>0,
\end{eqnarray*}
where $\gamma=\frac{\delta}{2(1-\sigma)(2+\alpha)}
      \sum\limits_{\omega_{N-1}\in\{1,2\}^{N-1}}\frac{p_{\omega_{N-1}}(x_{N-1,\omega_{N}})}{T'_{\omega_{N-1}}(x_{N-1,\omega_{N}})}$
      with $N,\sigma$ depending only on $A.$
Moreover, for $n>N,$ we set $h(x)=\mathcal{L}^{n-N}_{T}f(x).$ Then $h(x)\in\mathcal{C}_{A},$ and
\begin{eqnarray*}
\mathcal{L}^{n}_{T}f(x)&=&\mathcal{L}^{N}_{T}(\mathcal{L}^{n-N}_{T}f(x))=\mathcal{L}^{N}_{T}h(x)\geq\gamma.
\end{eqnarray*}
Thus, for all $n\geq N, \mathcal{L}^{n}_{T}f(x)\geq \gamma>0.$
For last part of the proposition, suppose that $\rho=\mathcal{L}_{T}\rho \in\mathcal{C}_{A}.$ Clearly,
if set $E$ such that $m(E)=0,$ it follows that $\mu(E)=\int\limits_{E}\rho dm=0.$
Conversely, $\mu(E)=0. \quad \rho=\mathcal{L}^{n}_{T}\rho$ implies that $0=\mu(E)=\int\limits_{E}\rho dm=\int\limits_{E}\mathcal{L}^{n}_{T}\rho dm\geq\gamma m(E).$
Hence, if $\rho=\mathcal{L}_{T}\rho$ then $\mu=\rho m$ is equivalent to $m.$
\end{proof}

\begin{proof}(Theorem \ref{thm3.1})
Since $\mathcal{C}_{A}$ is compact and convex, operator $\mathcal{L}_{T}$ has a fixed point
$f_{\ast}\in\mathcal{C}_{A}$ by Proposition \ref{pro3.4} and the Schauder-Tychonoff fixed point theorem of \cite{DS}. Thus,
random map $T$ admits an ACIM. Next, we give the proof of uniqueness.
Suppose that the random map $T$ has two mutually singular ACIM $\mu_{1}$ and $\mu_{2}.$ From Lemma 3.6, we have
$I=A_{k}\subseteq \textrm{supp}(\mu_{1})$ and $I=A_{k}\subseteq \textrm{supp}(\mu_{2}).$
Therefore, $I\subseteq \textrm{supp}(\mu_{1})\cap \textrm{supp}(\mu_{2}).$ This contradicts the mutual singularity
of $\mu_{1}$ and $\mu_{2}.$ Thus, the random map $T$ has a unique ACIM. By Proposition 3.7,
the invariant density $\rho$ is uniformly bounded below.
\end{proof}

\section{example}
We present an example of a random map $T$ which satisfies assumptions (A) and (B). Consequently, by Theorem \ref{thm3.1} this random map has a unique ACIM.

\begin{example}
\rm{ Let random map $T=\{T_{1}(x),T_{2}(x); p_{1}(x), p_{2}(x)\}$,
for $0<\beta<\alpha<1,$
$$
T_{1}=\begin{cases}
       x(1+2^{\alpha}x^{\alpha}) \quad x\in[0,\frac{1}{2}),\\
       2x-1 \quad \quad \quad x\in[\frac{1}{2},1].
       \end{cases}
\quad\quad
T_{2}=\begin{cases}
       x(1+2^{\beta}x^{\beta}) \quad x\in[0,\frac{1}{2}),\\
       \frac{3}{2}x-\frac{3}{4} \quad \quad \quad x\in[\frac{1}{2},1].
       \end{cases}
$$
and
$$
p_{1}=\begin{cases}
       \frac{1+x^{\alpha}}{3} \quad x\in[0,\frac{1}{2}),\\
       \frac{1}{3} \quad \quad \quad x\in[\frac{1}{2},1].
       \end{cases}
\quad\quad
p_{2}=\begin{cases}
       \frac{2-x^{\alpha}}{3} \quad x\in[0,\frac{1}{2}),\\
       \frac{2}{3} \quad \quad \quad x\in[\frac{1}{2},1].
       \end{cases}
.$$
We have $p_{1}(x),p_{2}(x)\in [0,1], p_{1}(x)+p_{2}(x)=1, \forall x\in[0,1]$
and $\inf\limits_{x\in I}p_{k}(x)\geq \frac{1}{3} >0.$ Thus, condition (B) is satisfied.
We now check condition (A).
%%$\sum\limits_{i=1}\limits^{l}\frac{p_{k}(T_{k,i}^{-1}(x))}{T^{'}_{k}(T_{k,i}^{-1}(x))}, 1\leq l\leq 2,$ is decreasing for all $k=1,2$.
First, for $x\in[0,\frac{1}{2}), \frac{p_{2}(x)}{T^{'}_{2}(x)}$ is obviously decreasing. For $\frac{p_{1}(x)}{T^{'}_{1}(x)},$ we show
$$p_{1}^{'}(x)T^{'}_{1}(x)-p_{1}(x)T^{''}_{1}(x)\leq 0, \forall x\in[0,\frac{1}{2}).$$
\begin{eqnarray*}
p_{1}^{'}(x)}{T^{'}_{1}(x)-p_{1}(x)T^{''}_{1}(x)
&=&\frac{1}{3}\alpha x^{\alpha-1}[1+(1+\alpha)2^{\alpha}x^{\alpha}]-\frac{1+x^{\alpha}}{3}[\alpha2^{\alpha}(1+\alpha)x^{\alpha-1}]\\
&=&\frac{\alpha x^{\alpha-1}}{3}[1+(1+\alpha)2^{\alpha}x^{\alpha}-(1+x^{\alpha})2^{\alpha}(1+\alpha)]\\
&=&\frac{\alpha x^{\alpha-1}}{3}[1-2^{\alpha}(1+\alpha)].
\end{eqnarray*}
The term in square bracket is negative, i.e. $1<2^{\alpha}(1+\alpha), \forall \alpha\in (0,1).$\\
So, $\frac{p_{1}(x)}{T^{'}_{1}(x)}$ is decreasing since
$$(\frac{p_{1}(x)}{T^{'}_{1}(x)})^{'}=\frac{p_{1}^{'}(x)T^{'}_{1}(x)-p_{1}(x)T^{''}_{1}(x)}{(T^{'}_{1}(x))^{2}}\leq 0.$$
For $x\in[\frac{1}{2},1], \frac{p_{1}(x)}{T^{'}_{1}(x)}=\frac{1}{6},  \frac{p_{2}(x)}{T^{'}_{2}(x)}=\frac{4}{9}.$ Therefore,
$\sum\limits_{i=1}\limits^{l}\frac{p_{k}(T_{k,i}^{-1}(x))}{T^{'}_{k}(T_{k,i}^{-1}(x))}, 1\leq l\leq 2,$ is decreasing for all $k=1,2,$
since $x\mapsto T^{-1}_{1,1}x$ and $x\mapsto T^{-1}_{2,1}x$ are increasing. This random map
preserves a unique ACIM.}
\end{example}

\section{two dimensional non-uniformly expanding map}
In this section we use the skew product representation of \cite{BBQ} and show that our random maps give rise to an interesting family of 2-dimensional non-uniformly expanding maps  which admit a unique ACIM. This family could serve as a good testing tool for the analysis of 2-dimensional systems with slow mixing.

Let $S(x,\omega): I^{2}\rightarrow I^{2}$ be
$$S(x,\omega)=(T_{k}(x), \varphi_{k,x}(\omega)), \text{ where }
\begin{cases}
\varphi_{1,x}(\omega)=\frac{\omega}{p_{1}(x)}, \quad\quad \omega\in[0,p_{1}(x)),\\
\varphi_{2,x}(\omega)=\frac{\omega-p_{1}(x)}{p_{2}(x)}, \quad  \omega\in[p_{1}(x),1].
\end{cases}
$$
%%%%In our problem there are four disjoint sets $U_{1},U_{2},U_{3},U_{4}$ .
Define $S_{i}=S\mid_{U_{i}}, i=1,2,3,4.$

$S_{1}=(T_{1,1}(x), \varphi_{1,x}(\omega))=(x(1+2^{\alpha}x^{\alpha}), \frac{\omega}{p_{1}(x)}),\quad \quad\quad U_{1}=[0,\frac{1}{2})\times[0,p_{1}(x)),$

$S_{2}=(T_{1,2}(x), \varphi_{1,x}(\omega))=(g_{1}(x),\frac{\omega}{p_{1}(x)}),\quad\quad \quad\quad\quad\quad\quad  U_{2}=[\frac{1}{2},1]\times[0,p_{1}(x)),$

$S_{3}=(T_{2,2}(x), \varphi_{2,x}(\omega))=(g_{2}(x), \frac{\omega-p_{1}(x)}{p_{2}(x)}),\quad\quad \quad\quad\quad  U_{3}=[\frac{1}{2},1]\times[p_{1}(x),1],$

$S_{4}=(T_{2,1}(x), \varphi_{2,x}(\omega))=(x(1+2^{\beta}x^{\beta}), \frac{\omega-p_{1}(x)}{p_{2}(x)}),\quad U_{4}=[0,\frac{1}{2})\times[p_{1}(x),1].$

One can easily check that $(0,0)$ is a fixed point of $S.$ Moreover, the lyapunov exponent in the horizontal direction has value zero at $(0,0).$
Therefore, $S$ is nonuniformly expanding map. Moreover, under conditions $(A)$ and $(B),$ since $T$ has unique ACIM, by \cite{BBQ}, $S$ has a unique ACIM too.

\section{stochastic stability}
In this section we study stochastic stability of random intermittent maps. For this purpose, we write, for $\varepsilon>0, 0<\alpha-\varepsilon<1,$
$$T_{\varepsilon}=\{T_{1}(x),T_{1,\varepsilon}(x); p_{1,\varepsilon}(x), p_{2,\varepsilon}(x)\},$$ where
$$
T_{1}=\begin{cases}
       x(1+2^{\alpha}x^{\alpha}) \quad x\in[0,\frac{1}{2}),\\
       g_{1}(x) \quad \quad \quad x\in[\frac{1}{2},1].
       \end{cases}
\quad\quad
T_{1,\varepsilon}=\begin{cases}
       x(1+2^{\alpha-\varepsilon}x^{\alpha-\varepsilon}) \quad x\in[0,\frac{1}{2}),\\
       g_{1,\varepsilon}(x) \quad \quad \quad x\in[\frac{1}{2},1].
       \end{cases}
.$$

Our main result in this section is the following theorem.
\begin{theorem}\label{thm5.1}
Let $f_{\varepsilon}$ be the unique invariant density of $T_{\varepsilon}.$ Let $f^{\ast}$ be the unique invariant density of $T_{1}.$
If $\lim\limits_{\varepsilon\rightarrow 0} \sup\limits_{x}p_{2,\varepsilon}(x)=0,$
then $\lim\limits_{\varepsilon\rightarrow 0}\|f_{\varepsilon}-f^{\ast}\|_{1}=0.$
\end{theorem}

\begin{proof}
Let $\mathcal{L}_{T_{\varepsilon}}$ be the Perron-Frobenius operator associated with the random map $T_{\varepsilon}.$ By Theorem \ref{thm3.1}, there exist a fixed point $f_{\varepsilon}$ of $\mathcal{L}_{T_{\varepsilon}}$
and $f_{\varepsilon}\in \mathcal{C}_{A},$ for some $A\geq\frac{4}{1-\alpha}.$ Since $\mathcal{C}_{A}$ is a compact set,
there exists a subsequence
$\{f_{\varepsilon_{k}}\}_{\varepsilon_{k}>0}$ of $\{f_{\varepsilon}\}_{\varepsilon>0}$ such that
$$f_{\varepsilon_{k}}\overset{L^{1}}\longrightarrow f^{\ast}\in \mathcal{C}_{A}, \text{ as } \varepsilon_{k}\rightarrow 0.$$
We have
\begin{eqnarray*}
\|f^{\ast}-P_{T_{1}}f^{\ast}\|_{1}&\leq&\|f^{\ast}-f_{\varepsilon_{k}}\|_{1}+\|f_{\varepsilon_{k}}-P_{T_{1}}f_{\varepsilon_{k}}\|_{1}
                                         +\|P_{T_{1}}f_{\varepsilon_{k}}-P_{T_{1}}f^{\ast}\|_{1}\\
&\leq&\|f^{\ast}-f_{\varepsilon_{k}}\|_{1}+\|f_{\varepsilon_{k}}-P_{T_{1}}f_{\varepsilon_{k}}\|_{1}
                                         +\|f_{\varepsilon_{k}}-f^{\ast}\|_{1}\\
&=& 2\|f^{\ast}-f_{\varepsilon_{k}}\|_{1}+\|\mathcal{L}_{T_{\varepsilon_{k}}}f_{\varepsilon_{k}}-P_{T_{1}}f_{\varepsilon_{k}}\|_{1}.                                     \end{eqnarray*}
The first term on the right converges to 0 as $\varepsilon_{k}\rightarrow 0$ by the choice of subsequence.
Moreover, we have
\begin{eqnarray*}
\|\mathcal{L}_{T_{\varepsilon_{k}}}f_{\varepsilon_{k}}-P_{T_{1}}f_{\varepsilon_{k}}\|_{1}
&=&\|P_{T_{1}}(p_{1,\varepsilon_{k}}f_{\varepsilon_{k}})+P_{T_{1,\varepsilon}}(p_{2,\varepsilon_{k}}f_{\varepsilon_{k}})
     -P_{T_{1}}f_{\varepsilon_{k}}\|_{1}\\
&=&\|P_{T_{1}}(p_{1,\varepsilon_{k}}f_{\varepsilon_{k}}-f_{\varepsilon_{k}})+P_{T_{1,\varepsilon}}(p_{2,\varepsilon_{k}}f_{\varepsilon_{k}})\|_{1}\\
&=&\|(P_{T_{1,\varepsilon}}-P_{T_{1}})p_{2,\varepsilon_{k}}f_{\varepsilon_{k}}\|_{1}\\
&\leq&2\|p_{2,\varepsilon_{k}}f_{\varepsilon_{k}}\|_{1}
\leq2\sup\limits_{x}p_{2,\varepsilon_{k}}\rightarrow 0, \text{ as } \varepsilon_{k}\rightarrow 0.
\end{eqnarray*}
Thus, $f^{\ast}=P_{T_{1}}f^{\ast} \quad m-$a.e.

% From the proof of Proposition 3.7, we can find an $N$ such that
%    \begin{eqnarray*}
%  \lim\limits_{\varepsilon\rightarrow 0}\mathcal{L}^{N}_{T_{\varepsilon}}f_{\varepsilon}(x)
%   &\geq&\lim\limits_{\varepsilon\rightarrow 0}\mathcal{L}^{N}_{T_{\varepsilon}}f_{\varepsilon}(1)\\
%&=&\lim\limits_{\varepsilon\rightarrow0}\sum\limits_{\omega_{N}\in\{1,2\}^{N}}\sum\limits_{i=1}\limits^{q}\frac{(p_{\omega_{N}}f)(b_{i})}{T'_{\omega_{N}}(b_{i})}\\
%    &\geq&\lim\limits_{\varepsilon\rightarrow 0}\frac{(p_{\omega_{N}=\{1\}^{N}}f)(b_{1})}{T'_{\omega_{N}=\{1\}^{N}}(b_{1})}\\
%    &\geq&\frac{1}{2(1-\sigma)(2+\alpha)^{N}}=c_{0}.
%    \end{eqnarray*}
%    For $n\geq N, \lim\limits_{\varepsilon\rightarrow 0}\mathcal{L}^{n}_{T_{\varepsilon}}f_{\varepsilon}(x)\geq c_{0}>0.$
%    Then $\lim\limits_{\varepsilon\rightarrow 0}f_{\varepsilon}(x)=\lim\limits_{\varepsilon\rightarrow 0}\mathcal{L}^{n}_{T_{\varepsilon}}f_{\varepsilon}(x)$
%   is bounded below uniformly. Thus, if $\mu=f^{\ast}m,$ then $\mu$ is an a.c.i.m for $T_{1}.$

By the uniqueness of $T_{1}-$ACIM, all
subsequences $\{f_{\varepsilon_{k_{i}}}\}_{\varepsilon_{k_{i}}>0}$ of $\{f_{\varepsilon}\}_{\varepsilon>0}$ have $f^{\ast}$ as their common limit point.
Hence, $\|f_{\varepsilon}-f^{\ast}\|_{1}\rightarrow 0,$ as $\varepsilon\rightarrow 0.$
\end{proof}

\section{Appendix}
Let
$$
\tau=\begin{cases}
       x(1+2^{\alpha}x^{\alpha}) \quad x\in[0,\frac{1}{2}),\\
       g(x) \quad \quad \quad x\in[\frac{1}{2},1].
       \end{cases}
$$

We study a deterministic map $\tau: I\rightarrow I,$
with partition $\mathcal{P}=\{I_{1},I_{2}\}, I_{1}=[0,\frac{1}{2}], I_{2}=[\frac{1}{2},1], g(\frac{1}{2})=0, g'(x)>1.$
\begin{lemma}
Let $\nu$ be a $\tau$ ACIM. Then the support of $\nu$ is $I.$
\end{lemma}

\begin{proof}
If $g(x)$ is onto, the uniqueness of the $\tau-$ACIM is well known (see \cite{LSV}).
We only consider the case, when $g(1)< 1.$ We have $\tau[0,\frac{1}{2}]=[0,1).$
We need to show that for any interval $J\subset I,$ there exists an $n\geq 1$ such that $\tau^{n}(J)\supseteq [0,\frac{1}{2}].$
If $J\supset I_{k}, k=1,2,$ then obviously $\tau(J)\supseteq [0,\frac{1}{2}].$ If $J\subset  I_{k}.$ Since $m(\tau(J))>m(J),$ there exists a
$j\geq 1$ such that $\tau^{j}(J)$ contains $\frac{1}{2}$ in its interior.

Since $\tau^{j}(J)$ contains the partition point $\frac{1}{2}$ in its interior, i.e. $\tau^{j}(J)\supset (t_{1}, t_{2})$
with $\frac{1}{2}\in (t_{1}, t_{2}).$ Then $\tau[\frac{1}{2}, t_{2}]=[g(\frac{1}{2}), g(t_{2})]=[0, g(t_{2})],$ which contains
the  point $0.$ Then obviously there exists a $l\geq 1$ such that
$\tau^{j+l}(J)\supseteq [0,\frac{1}{2}].$

Let $A$ denote the support of $\nu.$ Since $A$ contains an interval $J$ and $A$ is an invariant set $\tau^{n}(J)\subseteq A, n\geq 1$. Then,
$[0,1)\subset A.$
Consequently (by invariance) $A$ must contain $I.$ Moreover, $A\subset I.$ Therefore, the support of $\nu$ is $I.$
\end{proof}
\subsection*{Acknowledgement}
We would like to thank anonymous referee for useful comments. The referee's comments have greatly improved the presentation of the paper.
\bibliographystyle{amsplain}

\begin{thebibliography}{10}


\bibitem{A} Alves, J.: \textit{Strong statistical stability for robust classes of maps with non-uniform expanding maps},
Nonlinearity, \textbf{17}, (2004), no. 4, 1193-1215.

\bibitem{AV} Alves, J. and Viana, M. : \textit{Statistical stability for robust classes of maps with non-uniform expansion},
Ergodic Theory and Dynam. Systems, \textbf{22}, (2002), no. 1, 1-32.

\bibitem{AT} Araujo, V. and Tahzibi, A. : \textit{Stochastic stability at the boundary of expanding maps},
Nonlinearity, \textbf{18}, (2005), 1-20.

\bibitem{BBQ}  Bahsoun, W., Bose, C. and Quas, A. : \textit{Deterministic representation for position
dependent random maps}, Discrete Contin. Dynam. Systems,  \textbf{22}, (2008), no. 3, 529-540.

\bibitem{BG} Bahsoun, W. and G\'{o}ra, P. : \textit{Position dependent random maps in one and higher dimensions},
Studia Math., \textbf{166}, (2005), 271-286.

\bibitem{BG2} Bahsoun, W. and G\'{o}ra, P. : \textit{Weakly convex and concave random maps with position dependent probabilities},
 Stochastic Anal. Appl. \textbf{21}, (2003), no. 5, 983-994.

\bibitem{BV} Bahsoun, W. and Vaienti, S. : \textit{Metastability of certain intermittent maps},
Nonlinearity, \textbf{25}, (2012), 124.

\bibitem{ABG} Boyarsky, A. and G\'{o}ra, P. : \textit{Laws of Chaos: invariant measures and dynamical systems in
one dimension}, Birkh\"{a}user Boston, 1997.

\bibitem{DS}  Dunford, N. and Schwartz, J. : \textit{Linear operators Part I: general theory}, Interscience Publ, 1964.

\bibitem{GHW} Gonz\'{a}lez-Tokman, C., Hunt, B. and Wright, P. : \textit{Approximating invariant densities of metastable systems},
Ergodic Theory and Dynam. Systems \textbf{31}, (2011), 1345¨C1361.


\bibitem{HV} Hu, H. and Vaienti, S. : \textit{Absolutely continuous invariant measures for non-uniformly expanding maps},
Ergodic theory Dynam. System ,  \textbf{29}, (2009),1185-1215.

%\bibitem{K}  Kifer, Y: \textit{Ergodic Theorey of Random Transformations}, Birkh\"{a}user,
%Boston, Basel, Stuttgart, 1986.

\bibitem{LSV} Liverani, C., Saussol, B. and Vaienti, S. : \textit{A probabilistic approach to intermittency},
Ergodic theory Dynam. System, \textbf{19}, (1999),  671-685.

\bibitem{RM}  Murray, R. : \textit{Ulam's method for some non-unoformly expanding maps}, Discrete. Contin. Dyn. Syst.  \textbf{26}, (2010), no. 3, 1007-1018.

\bibitem{P} Pianigiani, G. : \textit{First return map and invariant measures},
Isr. J. Math., \textbf{35}, (1980), 32-48.

\bibitem{Y2}  Young, L-S :  \textit{Recurrence times and rates of mixing}, Isr. J. Math., \textbf{110}, (1999), 153-188.
\end{thebibliography}

\end{document}